\documentclass[11pt,a4paper]{article}

\usepackage[T1]{fontenc}
\usepackage{lmodern}
\usepackage{amsmath,amssymb,amsthm,mathtools}
\usepackage[margin=28mm]{geometry}
\usepackage{microtype}
\usepackage{graphicx}
\usepackage{placeins}
\usepackage[authoryear,round]{natbib}
\usepackage[hidelinks]{hyperref}
\hypersetup{
  pdftitle={Weighted Risk Aggregation under Long-Memory Dependence: Pathwise Fluctuations of Exposure-Weighted Losses},
  pdfauthor={Elina Moldavskaya},
  pdfsubject={Actuarial application of weighted long-memory LILs},
  pdfkeywords={actuarial risk aggregation, exposure-weighted losses, long-range dependence, law of the iterated logarithm, Hermite rank, Gaussian sequences}
}

\newtheorem{proposition}{Proposition}[section]
\newtheorem{corollary}[proposition]{Corollary}
\numberwithin{equation}{section}
\DeclareMathOperator{\Var}{Var}
\DeclareMathOperator{\Cov}{Cov}
\newcommand{\E}{\mathbb E}
\newcommand{\R}{\mathbb R}
\newcommand{\C}{\mathbb C}
\newcommand{\as}{\text{a.s.}}
\newcommand{\cC}{\mathcal C}
\newcommand{\HH}{\mathfrak H_{\mathbb R}}

\title{Weighted Risk Aggregation under Long-Memory Dependence:\\
Pathwise Fluctuations of Exposure-Weighted Losses}
\author{
Elina Moldavskaya\\[2pt]
\small Technion--Israel Institute of Technology, Israel\\
\small \texttt{elina.mol@technion.ac.il}
}
\date{}

\begin{document}
\maketitle

\begin{abstract}
We model insurance loss in each period as deterministic exposure multiplied by a nonnegative function of a long-memory Gaussian sequence. A weighted law of the iterated logarithm yields the exact almost-sure scale and sharp constant of centered cumulative loss. Loss per unit of cumulative exposure converges almost surely to its mean, with an exact error scale common to all admissible exposure profiles; the sharp constant depends on the exposure exponent. Lognormal and quadratic examples, of Hermite ranks one and two, and a simulation illustrate the results. A known spectral formula expresses the quadratic coefficient through a leading eigenvalue.
\end{abstract}

\medskip
\noindent\textbf{MSC2020 subject classifications.}
Primary 91G05; secondary 60F15, 60G10, 60G15.


\section{Introduction}
\label{sec:introduction}

Cumulative insurance losses reflect both the evolution of exposure and
the dependence of losses across accounting periods. Here risk aggregation
refers to accumulation over the accounting periods of a single portfolio.
A changing exposure
schedule modifies the contribution of each period, while persistent
systematic conditions can link observations separated by long time
intervals. These features appear in several actuarial settings.
\citet{Badescu2016} incorporate temporal dependence and exposure
fluctuations into a marked Cox model for claim reserving.
\citet{Wang2021} study actuarial valuation and hedging through a
Volterra mortality model with long-range dependence. Such models
motivate the study of accumulated fluctuations when a persistent risk
environment acts on a portfolio whose exposure changes over time.

How large are the recurrent deviations of cumulative loss from its
expectation along a single indefinitely observed trajectory? How rapidly
does the observed loss per unit of cumulative exposure stabilize around
its mean? These questions concern temporal risk aggregation and the
precision of an exposure-weighted estimate of mean loss. A law of the
iterated logarithm (LIL) specifies a normalizing sequence and an exact
almost-sure limsup coefficient. Together they describe an eventual
envelope and recurrent excursions arbitrarily close to its asymptotic level.

We index accounting periods by $t=1,2,\ldots$ and model the loss in
period $t$ as $C_t=e_tg(X_t)$, where $(X_t)_{t\geq1}$ is a
stationary standard Gaussian sequence with long memory (long-range
dependence), $g$ is a nonnegative loss
response with finite second moment, and $e_t>0$ is deterministic
exposure. The covariance satisfies
$\rho(h)\sim c h^{-\alpha}L_0(h)$, with $c>0$ and $0<\alpha<1$,
and the exposure has the form $e_t=t^{-d}\ell_e(t)$, with $d\in\R$.
The functions $L_0$ and $\ell_e$ are slowly varying. Writing
$\mu=\E g(X_t)$ and $V_n=\sum_{t=1}^n e_t$, the centered cumulative
loss is
\[
S_n=\cC_n-\mu V_n
   =\sum_{t=1}^n e_t\{g(X_t)-\mu\}.
\]
The Hermite rank $m$ of $g-\mu$ identifies the leading Gaussian-chaos
component of this sum. We work in the region
$\alpha m<1$ and $2d+\alpha m<1$, with the additional spectral and
smoothness assumptions required for nonlinear ranks stated in
Section~\ref{sec:model}.

The probabilistic background includes the long-memory LIL for
nonlinear Gaussian functionals studied by \citet{Taqqu1977} and the
functional LILs for multiple Wiener integrals established by
\citet{MoriOodaira1986,MoriOodaira1987}. The regularly weighted
framework of \citet{Moldavskaya2026} supplies the variance asymptotics
and exact LIL used here. We apply these results directly to the
exposure-weighted aggregate. Its recurrent absolute deviations have
scale
\[
b_n=n^{1-d-\alpha m/2}\ell_e(n)L_0(n)^{m/2}
       (2\log\log n)^{m/2},
\]
with a coefficient determined by the response, the covariance
amplitude, and the variance-normalized LIL constant
$\Lambda_m(\alpha,d)$. The power of the iterated-logarithm factor is
half the Hermite rank.

Dividing cumulative loss by cumulative exposure gives
$\widehat\mu_n=\cC_n/V_n$, an unbiased estimate of $\mu$ and a scalar,
zero-lag least-squares instance of the weighted empirical risk
minimization studied by \citet{MoldavskayaML2026}. We give a direct
proof of its consistency and exact LIL. Within the admissible class, normalization by $V_n$ yields the
error rate $n^{-\alpha m/2}L_0(n)^{m/2}(2\log\log n)^{m/2}$,
independent of the exposure profile. The exposure exponent $d$
affects the sharp constant.

Two nonnegative responses illustrate the results: a lognormal response
of Hermite rank one and a symmetric quadratic response of rank two.
Their explicit Hermite coefficients and covariance functions show how
the form of the response changes the aggregate and normalized-loss
scales. For the quadratic example, the spectral representation in
\citet{MoldavskayaRankTwo2026} expresses the fluctuation coefficient
through the largest eigenvalue of a positive integral operator.
Complementary high-rank asymptotics for the same constants, including
an explicit attractive log-gas limit, are developed in
\citet{MoldavskayaHighRank2026}. The contribution of the present paper is
the actuarial specialization of
the weighted LIL, its consequence for loss per unit of exposure, and
explicit calculations for these two response models.

Section~\ref{sec:model} specifies the model and assumptions.
Section~\ref{sec:aggregate-lil} gives the variance scale and
aggregate-loss LIL, and Section~\ref{sec:unit-exposure} treats loss
per unit of cumulative exposure. Section~\ref{sec:examples} develops
the two response models, the rank-two spectral coefficient and a
short simulation of lognormal losses.
Section~\ref{sec:discussion} explains the actuarial interpretation.
Appendix~\ref{app:constant} recalls the variational definition of the
LIL constant.

\section{Model and assumptions}
\label{sec:model}

\subsection{Losses and exposure}

Accounting periods are indexed by $t=1,2,\ldots$.
Let $(X_t)_{t\geq1}$ be a stationary Gaussian sequence with $\E X_t=0$ and
$\E X_t^2=1$. The variable $X_t$ represents a systematic risk factor
in period $t$. We model the loss per unit of exposure by
\begin{equation}
Y_t=g(X_t),\qquad
\mu=\E g(X_1),\qquad
G(x)=g(x)-\mu,
\label{eq:loss-response}
\end{equation}
where $g:\R\to[0,\infty)$ is measurable, belongs to
$L^2(\R,\varphi)$, and satisfies $\Var(g(X_1))>0$.
Here $\varphi$ denotes standard Gaussian probability measure.
These conditions imply $0<\mu<\infty$ and $\E G(X_t)=0$.

Let $e_t>0$ be the deterministic exposure in period $t$. Define
\begin{equation}
C_t=e_tY_t,\qquad
V_n=\sum_{t=1}^n e_t,\qquad
\cC_n=\sum_{t=1}^n C_t.
\label{eq:aggregates}
\end{equation}
Thus $C_t$ is the period loss, $V_n$ is cumulative exposure, and
$\cC_n$ is cumulative loss. Exposure is measured in a fixed unit, and
$Y_t$ is measured in monetary units per unit of exposure. Consequently,
\begin{equation}
\E\cC_n=\mu V_n,\qquad
S_n:=\cC_n-\mu V_n
    =\sum_{t=1}^n e_tG(X_t).
\label{eq:centered-loss}
\end{equation}
Although losses are nonnegative, their centered deviations $S_n$ may
have either sign.

The model describes temporal aggregation of losses generated by a
common systematic factor. The response $g$ specifies the loss per
unit of exposure, and the deterministic sequence $(e_t)$ specifies
the exposure profile. Their combination in
\eqref{eq:loss-response}--\eqref{eq:aggregates} determines period
losses and their cumulative development.
The prescribed exposure profile represents portfolio growth or run-off,
while stationarity of $X_t$ and the fixed response $g$ keep the
loss-per-exposure distribution, including its mean $\mu$, constant
across periods.

\subsection{Dependence, Hermite rank, and exposure weights}

Write $\rho(h)=\E[X_tX_{t+h}]$, with the even extension to negative
lags. Assume
\begin{equation}
\rho(h)\sim c\,h^{-\alpha}L_0(h),\qquad h\to\infty,
\qquad c>0,\quad 0<\alpha<1,
\label{eq:covariance}
\end{equation}
where $L_0$ is measurable, eventually positive, and slowly varying
at infinity. A function $L$ is slowly varying if
$L(rx)/L(x)\to1$ as $x\to\infty$ for every $r>0$.

We use the probabilists' Hermite polynomials
\[
H_q(x)=(-1)^q e^{x^2/2}\frac{d^q}{dx^q}e^{-x^2/2},
\]
so that, in particular, $H_1(x)=x$ and $H_2(x)=x^2-1$.
The centered response has the $L^2(\varphi)$ expansion
\begin{equation}
G(x)=\sum_{q=m}^{\infty}\frac{J_q}{q!}H_q(x),\qquad
J_q=\E[G(X_1)H_q(X_1)],\qquad J_m\ne0,
\label{eq:hermite}
\end{equation}
where $m\geq1$ is its Hermite rank. This first nonzero order
determines the leading covariance decay of the loss response.
Gaussian orthogonality gives
\begin{equation}
\Cov(Y_t,Y_{t+h})
=\sum_{q=m}^{\infty}\frac{J_q^2}{q!}\rho(h)^q
\sim\frac{J_m^2c^m}{m!}\,h^{-\alpha m}L_0(h)^m.
\label{eq:loss-covariance}
\end{equation}
Indeed, $\rho(h)\to0$ and
$\sum_{q\geq m}J_q^2/q!=\E G(X_1)^2<\infty$, so the sum over
$q>m$ is $O(|\rho(h)|^{m+1})$.

The exposure sequence is regularly varying:
\begin{equation}
e_t=t^{-d}\ell_e(t)>0,\qquad d\in\R,
\label{eq:exposure}
\end{equation}
where $\ell_e$ is measurable, slowly varying, and positive at the
positive integers. Throughout, we assume
\begin{equation}
\alpha m<1,\qquad 2d+\alpha m<1.
\label{eq:regime}
\end{equation}
The first restriction keeps the loss covariance in
\eqref{eq:loss-covariance} nonsummable. The second is the weighted
regime used in the LIL results below. In particular,
\[
d<\frac{1-\alpha m}{2}<\frac12.
\]
The strict inequalities in \eqref{eq:regime} specify the parameter
region throughout the article.

For pure power exposures, $\ell_e\equiv1$, the cases $d<0$, $d=0$,
and $d>0$ describe growing, constant, and declining exposure,
respectively. In the general case $d=0$ permits a nonconstant slowly
varying exposure schedule. A single exposure sequence $(e_t)_{t\geq1}$
is used at every observation horizon.

\subsection{Additional assumptions for nonlinear ranks}
\label{subsec:nonlinear-assumptions}

For $m\geq2$, we retain the additional hypotheses of the nonlinear
weighted LIL in \citet{Moldavskaya2026}. Assume that $(X_t)$ admits
a spectral density $f$ with the convention
\begin{equation}
\rho(h)=\int_{-\pi}^{\pi}e^{ih\lambda}f(\lambda)\,d\lambda,
\qquad
f(\lambda)\sim c_f|\lambda|^{\alpha-1}L_0(1/|\lambda|)
\quad(\lambda\to0),
\label{eq:spectral}
\end{equation}
where $c_f>0$ and $f$ is bounded on compact sets away from zero.
With the same $L_0$ in \eqref{eq:covariance} and \eqref{eq:spectral},
the constants satisfy $c=c_\alpha c_f$, where
$c_\alpha=2\Gamma(\alpha)\cos(\pi\alpha/2)$.
We also assume
\begin{equation}
f\in W^{1,1}([ -\pi,-\delta]\cup[\delta,\pi])
\quad\text{for every }\delta\in(0,\pi),
\label{eq:spectral-regularity}
\end{equation}
and that $\ell_e$ is eventually continuously differentiable with
\begin{equation}
\frac{x\ell'_e(x)}{\ell_e(x)}\longrightarrow0,
\qquad x\to\infty.
\label{eq:weight-smoothness}
\end{equation}
Condition \eqref{eq:spectral-regularity} requires an absolutely
continuous representative with integrable derivative on each of
the two intervals away from the origin. Pure power exposures satisfy
\eqref{eq:weight-smoothness}. For nonlinear ranks, the covariance
and spectral conditions are imposed jointly. The variance asymptotics
and the rank-one scalar LIL use the loss-response, covariance and
exposure assumptions of Sections~2.1--2.2.

All logarithms below are natural, and expressions involving
$\log\log n$ are considered for $n\geq3$.

\section{Fluctuations of accumulated losses}
\label{sec:aggregate-lil}

The identity \eqref{eq:centered-loss} places the actuarial aggregate
exactly in the setting of \citet{Moldavskaya2026}, with $a_t=e_t$.
The results in this section are direct actuarial specializations of
its variance asymptotics and scalar LILs.

Let
\begin{equation}
A_n^2=\Var(S_n),\qquad
I_{m,\alpha,d}=\int_0^1\!\int_0^1
x^{-d}y^{-d}|x-y|^{-\alpha m}\,dx\,dy,
\qquad
\sigma_m^2=\frac{J_m^2c^m}{m!}I_{m,\alpha,d}.
\label{eq:normalization}
\end{equation}
The constant $\sigma_m$ denotes the positive square root.

\begin{proposition}[Variance scale]
\label{prop:variance}
Under \eqref{eq:loss-response}, \eqref{eq:covariance},
\eqref{eq:hermite}, \eqref{eq:exposure}, and \eqref{eq:regime},
\begin{equation}
A_n\sim \sigma_m n^{1-d-\alpha m/2}
\ell_e(n)L_0(n)^{m/2}.
\label{eq:variance-asymptotic}
\end{equation}
Moreover,
\begin{equation}
I_{m,\alpha,d}
=\frac{2\,\mathrm B(1-d,1-\alpha m)}{2-2d-\alpha m},
\label{eq:beta-integral}
\end{equation}
where $\mathrm B$ is the beta function.
\end{proposition}

\begin{proof}
Apply Proposition~2.4 of \citet{Moldavskaya2026} to
\eqref{eq:centered-loss}. Its assumptions (A1), (A3), and (A4) are
precisely the covariance, Hermite-rank, and weight conditions above.
For \eqref{eq:beta-integral}, symmetry and the substitution $x=yu$
give
\[
I_{m,\alpha,d}
=2\int_0^1 y^{-d}\int_0^y x^{-d}(y-x)^{-\alpha m}\,dx\,dy
=2\,\mathrm B(1-d,1-\alpha m)
\int_0^1 y^{1-2d-\alpha m}\,dy.
\]
All integrals are finite under \eqref{eq:regime}, and evaluation of
the last integral proves the formula.
\end{proof}

For $m\geq2$, let $\Lambda_m(\alpha,d)$ be the variance-normalized
constant of Corollary~3.9 in \citet{Moldavskaya2026}. Its variational
definition is recalled in Appendix~\ref{app:constant}. Set
$\Lambda_1(\alpha,d)=1$.

\begin{corollary}[Weighted aggregate-loss LIL]
\label{cor:aggregate}
Under the assumptions of Proposition~\ref{prop:variance}, and also
those of Section~\ref{subsec:nonlinear-assumptions} when $m\geq2$,
\begin{equation}
\limsup_{n\to\infty}
\frac{|\cC_n-\mu V_n|}{A_n(2\log\log n)^{m/2}}
=\Lambda_m(\alpha,d)
\qquad\as.
\label{eq:aggregate-lil}
\end{equation}
Equivalently,
\begin{equation}
\limsup_{n\to\infty}
\frac{|\cC_n-\mu V_n|}
{n^{1-d-\alpha m/2}\ell_e(n)L_0(n)^{m/2}
 (2\log\log n)^{m/2}}
=\sigma_m\Lambda_m(\alpha,d)
\qquad\as.
\label{eq:aggregate-explicit}
\end{equation}
When $m=1$, replacing the absolute numerator in
\eqref{eq:aggregate-lil} by $\cC_n-\mu V_n$ gives limsup $1$ and
liminf $-1$.
\end{corollary}

\begin{proof}
For $m=1$, apply Theorem~3.3 of \citet{Moldavskaya2026} to
\eqref{eq:centered-loss}. For $m\geq2$, the additional conditions
\eqref{eq:spectral}--\eqref{eq:weight-smoothness} give assumptions
(A2), (A6), and (A5), respectively, of that paper; its Corollary~3.9 therefore
applies. Finally, replace $A_n$ by its asymptotically equivalent
expression in \eqref{eq:variance-asymptotic}.
\end{proof}

The factor $A_n$ is the standard-deviation scale, whereas
$(2\log\log n)^{m/2}$ describes the additional scale of recurrent
pathwise deviations. Its exponent is half the Hermite rank of the
centered response. For $m\geq2$, \eqref{eq:aggregate-lil} describes
the absolute deviations from the mean aggregate loss. For rank one,
the signed statement also identifies the upper and lower endpoints.

\section{Loss per unit of accumulated exposure}
\label{sec:unit-exposure}

Define
\begin{equation}
\widehat\mu_n=\frac{\cC_n}{V_n}
=\frac{\sum_{t=1}^n e_tY_t}{\sum_{t=1}^n e_t}.
\label{eq:estimator}
\end{equation}
This is the observed loss per unit of cumulative exposure and an
estimator of $\mu$, computed from period losses $C_t$ and exposures $e_t$.
Since the denominator is deterministic,
\begin{equation}
\E\widehat\mu_n=\mu,\qquad
\widehat\mu_n-\mu=\frac{S_n}{V_n}.
\label{eq:estimator-identity}
\end{equation}

The estimator also minimizes the weighted empirical contrast
$V_n^{-1}\sum_{t\le n}e_t\ell_0(\vartheta;Y_t)$, where
$\ell_0(\vartheta;y)=\vartheta^2/2-\vartheta y$ differs from squared
loss by a term independent of $\vartheta$. Its score at $\mu$ is
$-G(X_t)$, its curvature is one, and its $L^2$ envelopes on compact
parameter intervals follow from $\E Y_t^2<\infty$. For an interval
containing $\mu$ in its interior, Proposition~4.1 of
\citet{MoldavskayaML2026} gives eventual equality of its constrained
minimizer with $\widehat\mu_n$. Theorem~6.2 of that paper then yields
the following scalar, zero-lag specialization, with score coefficient
$\mathcal J=-J_m/m!$. We give the direct proof from the aggregate LIL.

\begin{corollary}[Consistency and pathwise estimation error]
\label{cor:estimator}
Under the assumptions of Corollary~\ref{cor:aggregate},
\begin{equation}
V_n\sim\frac{n^{1-d}\ell_e(n)}{1-d},
\label{eq:exposure-sum}
\end{equation}
and
\begin{equation}
\limsup_{n\to\infty}
\frac{|\widehat\mu_n-\mu|}
{n^{-\alpha m/2}L_0(n)^{m/2}(2\log\log n)^{m/2}}
=(1-d)\sigma_m\Lambda_m(\alpha,d)
\qquad\as.
\label{eq:estimator-lil}
\end{equation}
In particular, $\widehat\mu_n\to\mu$ almost surely.
\end{corollary}

\begin{proof}
Because $d<1$, Karamata's summation theorem for regularly varying
sequences yields \eqref{eq:exposure-sum}; see \citet{BGT1987}.
Combining it with Proposition~\ref{prop:variance} gives
\[
\frac{A_n}{V_n}
\sim(1-d)\sigma_m n^{-\alpha m/2}L_0(n)^{m/2}.
\]
Now use \eqref{eq:estimator-identity} and
\eqref{eq:aggregate-lil}. The denominator in
\eqref{eq:estimator-lil} is a positive regularly varying sequence
of index $-\alpha m/2<0$, and therefore tends to zero. The finite
limsup implies almost-sure convergence of the estimation error to zero.
\end{proof}

Normalization by cumulative exposure cancels the factor
$n^{1-d}\ell_e(n)$ from the leading exposure scale. Thus the entire
normalizing sequence in \eqref{eq:estimator-lil},
$n^{-\alpha m/2}L_0(n)^{m/2}(2\log\log n)^{m/2}$, is independent of
$d$ and $\ell_e$. Exposure weighting still affects the sharp constant
through $d$, via $1-d$, $I_{m,\alpha,d}$, and, for nonlinear ranks,
$\Lambda_m(\alpha,d)$; the slowly varying factor $\ell_e$ cancels
completely at this order.

\section{Two nonnegative loss-response models}
\label{sec:examples}

The following examples have the same prescribed mean $\mu>0$ but
different responses to the systematic factor. They illustrate how
the response changes both the effective memory exponent and the
iterated-logarithm factor. Both are theoretical models designed
to make these effects explicit.

\subsection{Lognormal response: Hermite rank one}

Let
\begin{equation}
g_{\mathrm L}(x)=\mu\exp(\beta x-\beta^2/2),\qquad\beta>0.
\label{eq:lognormal-model}
\end{equation}
The response is positive and increasing in $x$. Gaussian integration
gives
\[
\E g_{\mathrm L}(X_t)=\mu,\qquad
\E g_{\mathrm L}(X_t)^2=\mu^2e^{\beta^2}<\infty.
\]
The Hermite generating identity yields
\[
G_{\mathrm L}(x)
=\mu\sum_{q=1}^{\infty}\frac{\beta^q}{q!}H_q(x),
\qquad J_q=\mu\beta^q.
\]
Thus the lognormal losses have Hermite rank one, with leading
coefficient $J_1=\mu\beta$. The exact covariance is
\begin{equation}
\Cov(Y_t,Y_{t+h})
=\mu^2\bigl(e^{\beta^2\rho(h)}-1\bigr)
\sim\mu^2\beta^2c\,h^{-\alpha}L_0(h).
\label{eq:lognormal-covariance}
\end{equation}
It follows either by applying the Gaussian exponential-moment
formula to $(X_t,X_{t+h})$ or by summing
\eqref{eq:loss-covariance}.

Under \eqref{eq:covariance} and \eqref{eq:exposure}, with
$2d+\alpha<1$, the variance coefficient is
\[
\sigma_1=\mu\beta\sqrt{cI_{1,\alpha,d}}.
\]
Corollary~\ref{cor:aggregate} gives
\begin{equation}
\limsup_{n\to\infty}
\frac{|\cC_n-\mu V_n|}
{n^{1-d-\alpha/2}\ell_e(n)L_0(n)^{1/2}
 \sqrt{2\log\log n}}
=\mu\beta\sqrt{cI_{1,\alpha,d}}
\qquad\as,
\label{eq:lognormal-lil}
\end{equation}
while Corollary~\ref{cor:estimator} gives
\begin{equation}
\limsup_{n\to\infty}
\frac{|\widehat\mu_n-\mu|}
{n^{-\alpha/2}L_0(n)^{1/2}\sqrt{2\log\log n}}
=(1-d)\mu\beta\sqrt{cI_{1,\alpha,d}}
\qquad\as.
\label{eq:lognormal-estimator}
\end{equation}
The covariance and exposure assumptions in Sections~2.1--2.2
suffice for this rank-one specialization.

\subsection{Symmetric quadratic stress: Hermite rank two}

Let
\begin{equation}
g_{\mathrm Q}(x)=\mu\bigl((1-\theta)+\theta x^2\bigr),
\qquad 0<\theta\leq1.
\label{eq:quadratic-model}
\end{equation}
Here $\mu(1-\theta)$ is a nonnegative baseline component, and the
remaining component increases with the magnitude of the systematic
factor, irrespective of its sign. Direct calculation gives
\[
\E g_{\mathrm Q}(X_t)=\mu,\qquad
G_{\mathrm Q}(x)=\mu\theta H_2(x),\qquad
J_2=2\mu\theta.
\]
All other centered Hermite coefficients vanish, so $m=2$.
Moreover,
\[
\E g_{\mathrm Q}(X_t)^2=\mu^2(1+2\theta^2)<\infty,
\]
and
\begin{equation}
\Cov(Y_t,Y_{t+h})=2\mu^2\theta^2\rho(h)^2
\sim2\mu^2\theta^2c^2h^{-2\alpha}L_0(h)^2.
\label{eq:quadratic-covariance}
\end{equation}
The centered aggregate is exactly
\[
S_n=\mu\theta\sum_{t=1}^n e_tH_2(X_t).
\]

Assume the conditions of Section~\ref{sec:model}, including the
additional nonlinear assumptions, with
\[
0<\alpha<\frac12,\qquad 2d+2\alpha<1.
\]
Then
\[
\sigma_2=\sqrt{2}\,\mu\theta c\sqrt{I_{2,\alpha,d}},
\]
and the two corollaries imply
\begin{equation}
\limsup_{n\to\infty}
\frac{|\cC_n-\mu V_n|}
{n^{1-d-\alpha}\ell_e(n)L_0(n)(2\log\log n)}
=\sqrt{2}\,\mu\theta c\sqrt{I_{2,\alpha,d}}\,
 \Lambda_2(\alpha,d)
\qquad\as,
\label{eq:quadratic-lil}
\end{equation}
and
\begin{equation}
\limsup_{n\to\infty}
\frac{|\widehat\mu_n-\mu|}
{n^{-\alpha}L_0(n)(2\log\log n)}
=(1-d)\sqrt{2}\,\mu\theta c\sqrt{I_{2,\alpha,d}}\,
 \Lambda_2(\alpha,d)
\qquad\as.
\label{eq:quadratic-estimator}
\end{equation}

\subsection{What the comparison shows}

For a common Gaussian factor and a common exposure schedule, both
examples lie in the present regime when
$0<\alpha<1/2$ and $2d+2\alpha<1$. The lognormal response retains
covariance decay of order $h^{-\alpha}L_0(h)$, whereas the quadratic
response has decay of order $h^{-2\alpha}L_0(h)^2$. Correspondingly,
the aggregate LIL scales contain the powers $n^{1-d-\alpha/2}$
and $n^{1-d-\alpha}$, respectively.

The comparison of aggregate deviations therefore uses the entire
normalization: the polynomial power, the slowly varying factor,
the iterated-logarithm factor and the multiplicative constant.
The two models have the same mean, while their marginal variances
are $\mu^2(e^{\beta^2}-1)$ and $2\mu^2\theta^2$, respectively.

\subsection{Spectral representation of the quadratic-loss coefficient}
\label{subsec:quantitative-constants}

The spectral analysis of \citet{MoldavskayaRankTwo2026} relates
the deterministic coefficient in \eqref{eq:quadratic-lil} to a
positive integral operator. This representation applies throughout
the admissible rank-two parameter region.

For $0<\alpha<1/2$ and $2d+2\alpha<1$, define the positive
self-adjoint Hilbert--Schmidt operator on $L^2(0,1)$ by
\begin{equation}
(T_{\alpha,d}\psi)(x)
=\int_0^1 x^{-d/2}y^{-d/2}|x-y|^{-\alpha}\psi(y)\,dy,
\qquad
\lambda_{\alpha,d}=\lambda_{\max}(T_{\alpha,d}).
\label{eq:ranktwo-operator}
\end{equation}
The rank-two spectral identity
\citep[Sections~2--3]{MoldavskayaRankTwo2026} is
\begin{equation}
\Lambda_2(\alpha,d)
=\frac{\lambda_{\alpha,d}}{\sqrt{2I_{2,\alpha,d}}}.
\label{eq:ranktwo-spectral-constant}
\end{equation}
Substitution of $\sigma_2$ therefore gives the simple coefficient
\begin{equation}
K_{\mathrm Q}:=\sigma_2\Lambda_2(\alpha,d)
=\mu\theta c\,\lambda_{\alpha,d}.
\label{eq:quadratic-coefficient}
\end{equation}
Thus the right-hand sides of \eqref{eq:quadratic-lil} and
\eqref{eq:quadratic-estimator} equal $K_{\mathrm Q}$ and
$(1-d)K_{\mathrm Q}$, respectively. For a prescribed response amplitude
and covariance coefficient, the remaining dependence on the memory
and exposure parameters is encoded by one leading eigenvalue.

Spectral evaluation methods for this coefficient are developed in
\citet{MoldavskayaRankTwo2026}.

\subsection{A lognormal illustration with fractional Gaussian noise}
\label{subsec:simulation}

For illustration, take $X_t=B_{H_0}(t)-B_{H_0}(t-1)$, where $B_{H_0}$
is fractional Brownian motion with $H_0=0.8$ and $\Var(B_{H_0}(1))=1$.
Then $\alpha=0.4$, $c=0.48$ and $L_0\equiv1$. We set $\mu=1$,
$\beta=0.5$ and $e_t=t^{-d}$ for $d=0,0.2$, both satisfying
$2d+\alpha<1$. These are illustrative parameter choices. We generate
$M=500$ independent paths of length $N=2^{18}$ at unit time steps by
FFT circulant embedding \citep{DaviesHarte1987}.
This method preserves the prescribed
finite Gaussian covariance, up to rounding, with $O(N\log N)$ cost
per path; all embedding eigenvalues are checked for nonnegativity.
Each Gaussian path is shared by the two exposure schedules.

Figure~\ref{fig:lognormal-fgn} shows the first replication, selected
in advance. It plots
\[
Z_{n,d}=\frac{\widehat\mu_n-\mu}
{k_d n^{-\alpha/2}\sqrt{2\log\log n}},\qquad
k_d=(1-d)\mu\beta\sqrt{cI_{1,\alpha,d}},
\]
whose signed LIL endpoints are $-1$ and $1$. For the selected values
$d=0$ and $d=0.2$, the asymptotic rank-one coefficients $k_d$ differ by
approximately $0.024\%$; the figure therefore emphasizes normalized
path behavior under the two exposure profiles. The running sums use all
observations; the plot uses $n_j=\operatorname{round}(2^{8+j/64})$,
$j=0,\ldots,640$. Across the 500 replications, the root mean squared
errors at $N$ are $0.0420$ and $0.0419$ for $d=0$ and $d=0.2$,
respectively. The exact finite-horizon standard deviations computed
from \eqref{eq:lognormal-covariance} are $0.0413$ for both schedules
at this precision; the discrepancies lie within Monte Carlo
variability.
The terminal RMSE and exact standard deviation quantify finite-horizon
dispersion and provide a numerical consistency check for the simulation.

\begin{figure}[htbp]
\centering
\includegraphics[width=\linewidth]{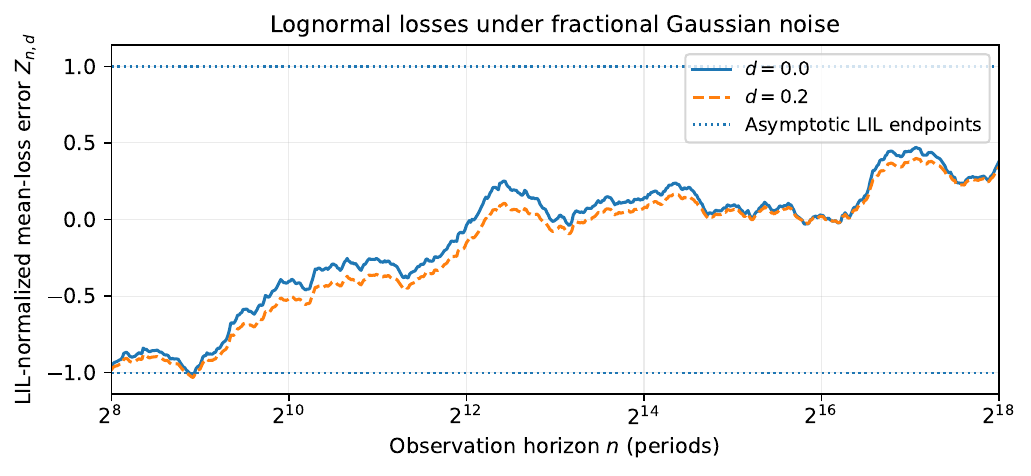}
\caption{LIL-normalized error of the exposure-weighted mean in the
lognormal model driven by fractional Gaussian noise. The two curves
use the same first realization with $H_0=0.8$, $\mu=1$ and $\beta=0.5$;
$d=0$ and $d=0.2$ specify constant and declining exposure.
The horizontal levels $\pm1$ are the asymptotic signed LIL endpoints.
The 641 plotted horizons range from $2^8$ to $2^{18}$.}
\label{fig:lognormal-fgn}
\end{figure}

\FloatBarrier
\section{Actuarial interpretation}
\label{sec:discussion}

The aggregate LIL quantifies recurrent deviations from the expected
loss $\mu V_n$. Let $b_n$ be the denominator in
\eqref{eq:aggregate-explicit} and put
$K=\sigma_m\Lambda_m(\alpha,d)>0$.
For every fixed $\varepsilon\in(0,K)$,
\eqref{eq:aggregate-explicit} implies, with probability one,
\[
|\cC_n-\mu V_n|\leq(K+\varepsilon)b_n
\quad\text{for all sufficiently large }n,
\]
and
\[
|\cC_n-\mu V_n|>(K-\varepsilon)b_n
\quad\text{for infinitely many }n.
\]
These statements describe one loss trajectory over an infinite
horizon, rather than exceedance probabilities at a fixed horizon.
The eventual bound begins at a path-dependent time.

The aggregate and normalized-loss results connect accumulated risk
with estimation of mean loss. For the quadratic response, the
spectral formula in Section~\ref{subsec:quantitative-constants}
expresses the envelope coefficient through one leading eigenvalue.

\appendix
\section{The variance-normalized LIL constant}
\label{app:constant}

For completeness, this appendix records the constant used in
Corollary~\ref{cor:aggregate}. It is the constant already established
in \citet[Section~3 and Corollary~3.9]{Moldavskaya2026}.

For $m\geq2$, let
\[
\HH=\{h\in L^2(\R;\C):h(-x)=\overline{h(x)}
\text{ for almost every }x\},\qquad
\langle h,k\rangle_{\HH}
=\operatorname{Re}\int_{\R}h(x)\overline{k(x)}\,dx.
\]
Symmetric tensor products are taken over this real Hilbert space.
Put $c_\alpha=2\Gamma(\alpha)\cos(\pi\alpha/2)$ and define the
symmetric $L^2(\R^m)$ kernel
\begin{equation}
Q_{m,\alpha,d}(x_1,\ldots,x_m)
=\frac{\displaystyle
\left(\int_0^1 u^{-d}e^{iu(x_1+\cdots+x_m)}\,du\right)
\prod_{j=1}^m |x_j|^{(\alpha-1)/2}}
{\sqrt{m!I_{m,\alpha,d}}\,c_\alpha^{m/2}}.
\label{eq:kernel}
\end{equation}
This normalization gives $m!\|Q_{m,\alpha,d}\|_2^2=1$. The constant
is
\begin{equation}
\Lambda_m(\alpha,d)
=\sup_{\|\xi\|_{\HH}\leq1}
\left|\left\langle Q_{m,\alpha,d},\xi^{\otimes m}
\right\rangle_{\HH^{\otimes m}}\right|.
\label{eq:lambda-definition}
\end{equation}
The source kernel at time one is
$\operatorname{sgn}(J_m)Q_{m,\alpha,d}$. The absolute value in
\eqref{eq:lambda-definition} makes both kernel conventions yield
the same constant. Thus $\Lambda_m$ depends on $m$, $\alpha$ and
$d$, while $\sigma_m$ carries the response amplitude. For rank one,
$\Lambda_1=1$ follows from the scalar theorem used in
Corollary~\ref{cor:aggregate}.

\end{document}